\documentclass[a4paper, twoside, reqno]{amsart}

\usepackage{amsmath}
\usepackage{amsthm}
\usepackage{amscd}
\usepackage{amssymb}
\usepackage{array}
\usepackage{enumerate}
\usepackage[pdftex]{graphicx}
\usepackage{hyperref}
\usepackage[latin1]{inputenc}
\usepackage{latexsym}
\usepackage{mathdots}
\usepackage{multirow}
\usepackage{stmaryrd}
\usepackage{tabularx}
\usepackage[all]{xy}

\newtheorem{theorem}{Theorem}[section]
\newtheorem{lemma}[theorem]{Lemma}
\newtheorem{corollary}[theorem]{Corollary}
\newtheorem{definition}[theorem]{Definition}
\newtheorem{proposition}[theorem]{Proposition}

\newtheorem*{blackadarconjecture}{Blackadar's Conjecture}

\newcommand{\calC}{\mathcal{C}}

\newcommand{\calF}{\mathcal{F}}
\newcommand{\calG}{\mathcal{G}}

\newcommand{\calO}{\mathcal{O}}
\newcommand{\calP}{\mathcal{P}}
\newcommand{\calU}{\mathcal{U}}

\newcommand{\bbK}{\mathbb{K}}

\newcommand{\bbN}{\mathbb{N}}
\newcommand{\bbT}{\mathbb{T}}
\newcommand{\bbZ}{\mathbb{Z}}

\newcommand{\Ad}[0]{\operatorname{Ad}}

\newcommand{\Bott}[0]{\operatorname{Bott}}
\newcommand{\cel}[0]{\operatorname{cel}}
\newcommand{\dist}[0]{\operatorname{dist}}

\newcommand{\Ext}[0]{\operatorname{Ext}}
\newcommand{\Hom}[0]{\operatorname{Hom}}
\newcommand{\id}[0]{\operatorname{id}}

\newcommand{\Length}[0]{\operatorname{Length}}
\newcommand{\rank}[0]{\operatorname{rank}}

\newcommand{\oline}{\overline}
\newcommand{\uline}{\underline}

\title[Semiprojectivity for Kirchberg algebras]{Semiprojectivity for Kirchberg algebras}

\date{\today}

\author{Dominic Enders}
\thanks{This work was supported by the Danish National Research Foundation through the {\itshape Centre for Symmetry and Deformation} (DNRF92).}

\address{Dominic Enders
\newline Department of Mathematical Sciences, University of Copenhagen
\newline Universitetsparken 5, DK-2100 Copenhagen \O, Denmark}
\email{d.enders@math.ku.dk}

\begin{document}

\begin{abstract}
We show that a Kirchberg algebra is semiprojective if and only if it is $KK$-semiprojective.
In particular, this shows that a Kirchberg algebra in the UCT-class is semiprojective if and only if its $K$-theory is finitely generated, thereby giving a positive solution to a longstanding conjecture by Blackadar.
\end{abstract}

\maketitle


\section{Introduction}\label{sec: 1}

Semiprojectivity is a rigidity property for $C^*$-algebras which is conceptually and technically important, as it gives the right framework for the systematic study of $C^*$-algebraic perturbation questions.
In this paper we identify which Kirchberg algebras, i.e., which separable, purely infinite, simple, nuclear $C^*$-algebras have the property of being semiprojective.

This problem was first studied in the mid-90's by Blackadar in the context of the deep classification results by Kirchberg and Phillips (\cite{Kir},\cite{Phi00}).
In fact, using said classification results, Blackadar showed that semiprojective Kirchberg algebras satisfying the Universal Coefficient Theorem of Rosenberg and Schochet (\cite{RS87}) must necessarily have finitely generated $K$-groups.
On the other hand, he was able to verify semiprojectivity in certain special cases, in particular for Cuntz-Krieger algebras $\calO_A$  and for the Cuntz algebra $\calO_\infty$.
Motivated by these examples, he conjectured that finitely generated $K$-theory should be the only obstruction for UCT-Kirchberg algebras to be semiprojective.

\begin{blackadarconjecture}[{\cite[Conj. 3.6]{Bla04}}]\label{conj:Blackadar}
 A Kirchberg algebra in the UCT-class is semiprojective if and only if its $K$-theory is finitely generated.
\end{blackadarconjecture}

Subsequently, Szyma{\'n}ski (\cite{Szy02}) confirmed semiprojectivity for UCT-Kirchberg algebras $A$ with finitely generated $K$-groups under the additional assumptions that $K_1(A)$ is torsion-free and $\rank(K_1(A))\leq\rank(K_0(A))$.
The restriction on the rank of the $K$-groups was then later removed by Spielberg (\cite{Spi09}).
Furthermore, a weak version of Blackadar's conjecture, i.e.\ a $K$-theoretic characterization of weak semiprojectivity for UCT-Kirchberg algebras, was independently obtained by Spielberg (\cite{Spi07}) and Lin (\cite{Lin07}).
However, despite significant effort by various people, Blackadar's conjecture remained open for UCT-Kirchberg algebras with torsion in $K_1$.
In this article, we give a $KK$-theoretic characterization of semiprojectivity for Kirchberg algebras which in the presence of the UCT reduces to a positive solution of Blackadar's conjecture.

The approach to verify semiprojectivity which we develop here is new and of independent interest.
In fact, it is completely different from the usual method. 
Typically, one chooses models for the $C^*$-algebras in question and then tries to find presentations of those by generators and stable relations.
In particular, all the above-mentioned papers follow this pattern.
More precisely, they use certain graph $C^*$-algebras as models and then verify that the relations coming from the underlying graph structure are partially liftable.
Kirchberg algebras with torsion in $K_1$, however, can not be realized as graph $C^*$-algebras.
Instead one needs to pass to higher rank graphs in the sense of \cite{KP00}, which leads to commutation relations in the presentation of these algebras.
Similarly, all other known presentations of Kirchberg algebras with torsion in $K_1$, typically coming from dynamical systems or more general crossed products, contain at least implicitly some sort of commutation relation.
Since commutation relations tend to be unstable in general, this explains the difficulty in dealing with $K_1$-torsion.

In significant contrast to this, our approach does not rely on the choice of any presentation by generators and relations at all.
Instead we study the class of $^*$-homomorphisms between Kirchberg algebras which enjoy good lifting properties.
Our main result here is a closure result, saying that semiprojectivity of $^*$-homo\-morphisms between Kirchberg algebras is preserved under asymptotic unitary equivalence. 
This reduces the question of semiprojectivity to finding semiprojective endomorphisms in the $KK$-class of the identity.
We then verify the existence of such endomorphisms for Kirchberg algebras satisfying a $KK$-theoretic version of semiprojectivity, which in the presence of the UCT is equivalent to having finitely generated $K$-theory.

The closure result mentioned above is based on the crucial observation that it is possible to lift an asymptotic unitary equivalence of $^*$-homomorphisms, provided one has control over the rectifiable length of the path of unitaries implementing it.
While this is not automatically the case, we exploit a Basic Homotopy Lemma as a method to suitably shorten the length of such paths.
The respective version of the Basic Homotopy Lemma for Kirchberg algebras is derived in Section \ref{sec: 2}, mostly following results from \cite{DE07}, while the lifting procedure is discussed in Section \ref{sec: 3}. 

In section \ref{sec: 4}, we deduce the existence result for semiprojective endomorphisms of Kirchberg algebras from Kirchberg's Classification Theorem for properly infinite $C^*$-algebras (\cite{Kir}). This part is largely inspired by Neub{\"u}ser's work on asymptotic semiprojectivity in \cite{Neu00}.

We conclude by combining the existence and closure results.
This yields our main Theorem \ref{thm: main}: For Kirchberg algebras, semiprojectivity and $KK$-semiprojectivity are equivalent.
If one further assumes the UCT, one obtains a positive answer to Blackadar's Conjecture: A UCT-Kirchberg algebra is semiprojective if and only if its $K$-theory is finitely generated.


\section{A Basic Homotopy Lemma}\label{sec: 2}

In this section, we provide a version of the Basic Homotopy Lemma which in particular covers inclusions $A\hookrightarrow B$ of UCT-Kirchberg algebras.
To our best knowledge, the first such result was obtained by Bratteli, Elliott, Evans and Kishimoto in \cite{BEEK98} for the case $A=B=\calO_n$.
It is possible to generalize their methods to general UCT-Kirchberg algebras using the models provided by Elliott and R{\o}rdam in \cite{ER95}.
In fact, this provided our first proof of Theorem \ref{thm: shortening} before we became aware of \cite[Theorem 2.6]{DE07}.

Here we follow the more advanced proof of \cite[Theorem 2.6]{DE07} with only a few minor modifications.
For instance, we name the appearing $KK$-theoretic obstacle and note that the unitary homotopies can be obtained with uniformly bounded length.
As a more significant change, we make use of the following $KK$-theoretic lifting property introduced in \cite[Section 3]{Dad09} in order to avoid any UCT-assumption.

\begin{definition}\label{def: KK-semiprojectivity}
A separable $C^*$-algebra $A$ is $KK$-semiprojective if for any separable $C^*$-algebra $D$ and any increasing sequence of ideals $J_n$ in $D$ with $J_\infty=\oline{\bigcup_n J_n}$, the induced map $\varinjlim KK(A,D/J_n)\to KK(A,D/J_\infty)$ is surjective.
\end{definition}

$KK$-semiprojectivity of a $C^*$-algebra $A$ is in fact equivalent to continuity of the functor $KK(A,-)$ with respect to countable direct limits, see \cite[Theorem 3.12]{Dad09} for a proof based on a mapping telescope construction.
Most importantly, this is the case whenever the $C^*$-algebra $A$ satisfies the UCT and has finitely generated $K$-theory, cf.\ \cite[Theorem 1.14, Proposition 7.13]{RS87}.

\begin{lemma}\label{lem: suspension}
Let $A$ be a separable $KK$-semiprojective $C^*$-algebra $A$. 
Then $C(\bbT)\otimes A$ is $KK$-semiprojective as well.
\end{lemma}

\begin{proof}
We only prove $KK$-semiprojectivity for the suspension of $A$, the claim then follows from split-exactness of $KK$.
Let a separable $C^*$-algebra $D$ and an increasing sequence of ideals $J_n$ in $D$ with $J_\infty=\oline{\bigcup_n J_n}$ be given.
We have the following diagram
\[
\xymatrix{
\varinjlim KK(SA,D/J_n) \cong \varinjlim KK(A,S(D/J_n)) \ar@<-1.5cm>[d] \ar@<1.5cm>[d] \\
KK(SA,D/J_\infty) \cong KK(A,S(D/J_\infty))
}
\]
with the horizontal isomorphisms obtained from Bott periodicity.
By naturality thereof, the diagram commutes.
As $A$ is $KK$-semiprojective, the map on the right is surjective, hence so must be the one on the left.
\end{proof}

\begin{lemma}\label{lem: technical}
Let $A$ be a separable, unital, nuclear, simple, $KK$-semiprojective $C^*$-algebra.
Then for every finite subset $\calF\subset A$ and every $\epsilon>0$, there exist a finite subset $\calG\subset A$ and $\delta>0$ satisfying the following:
For any unital Kirchberg algebra $B$, any unital $^*$-homomorphism $\varphi\colon A\to B$ and any unitary $u\in B$ with 
\[
 \|[\varphi(x),u]\|<\delta\quad\text{for all}\ x\in\calG,
\]
there exists a unital $^*$-homomorphism $\phi\colon C(\bbT)\otimes A\to B$ such that
\[
\|\phi(z\otimes 1)-u\|<\epsilon\quad\text{and}\quad\|\phi(1\otimes x)-\varphi(x)\|<\epsilon\quad\text{for all}\ x\in\calF.
\]
\end{lemma}

\begin{proof}
If the spectrum of the unitary $u$ is sufficiently dense in $\bbT$, the proof is basically identical to the one given in Theorem 2.6 of \cite{DE07} (which actually produces a proper $^*$-homomorphism $\phi$ in this case).
The only difference being that one uses $KK$-semiprojectivity of $C(\bbT)\otimes A$, i.e.\ Lemma \ref{lem: suspension}, instead of invoking the UCT.

For the general case, one can still proceed as in \cite{DE07} but has to replace $C(\bbT)\otimes A$ by $C(K)\otimes A$ with $K$ a finite union of (possibly degenerate) closed subintervals of $\bbT$.
All arguments carry over with only slight modifications, again replacing the UCT-assumptions by $KK$-semiprojectivity of $C(K)\otimes A$.
This gives a $^*$-homomorphism $\phi$ with the desired properties (which factors through $C(K)\otimes A$). 
\end{proof}

The next proposition ensures that a unitary can be connected to the identity in a way almost commuting with the image of a given $^*$-homomorphism, provided a certain $KK$-theoretic obstacle vanishes.
We are moreover interested in the length of such a path.
For this purpose, recall the notion of exponential length of a unital $C^*$-algebra $A$ from \cite{Rin92} which is defined as $\cel(A)=\sup_{u\in\calU_0(A)}\cel(u)$ with $\cel(u)$, the exponential length of a unitary $u\in\calU_0(A)$, given by
\[
\cel(u)=\inf\left\{\sum_{k=1}^n\left\|h_k\right\|\colon h_k\in A_{sa}, \quad u=e^{ih_1}e^{ih_2}\cdot ... \cdot e^{ih_n}\right\}.
\]
As shown in \cite{Rin92}, $\cel(u)$ coincides with the infimum of the lengths of rectifiable paths from $u$ to $1_A$ in $\calU(A)$, i.e.\ $\cel(A)$ describes the rectifiable diameter of $\calU_0(A)$.

\begin{proposition}\label{prop: KK-obstacle}
Let $A$ be a separable, unital, nuclear, simple $C^*$-algebra and $B$ a unital Kirchberg algebra.
Given a unital $^*$-homomorphism $\varphi\colon C(\bbT)\otimes A\to B$, the following holds:
If $[\varphi|_{SA}]=0$ in $KK(SA,B)$, then for any finite subset $\calF\subset A$ and $\epsilon>0$, there exists a rectifiable continuous path of unitaries $(u_t)_{t\in[0,1]}$ in $B$ with $u_0=1_B$ and $u_1=\varphi(z\otimes 1_A)$, such that
\[
  \|[\varphi(1\otimes x),u_t]\|<\epsilon\quad\text{for all}\ x\in\calF\ \text{and}\ t\in[0,1]
 \]
 and
 \[
  \Length\left((u_t)_{t\in[0,1]}\right)\leq 2\pi+\epsilon.
 \]
\end{proposition}

\begin{proof}
We outline the proof given in the second part of \cite[Theorem 2.6]{DE07}:
We may assume that $\varphi$ is injective.
 Fix an identification $B\cong B\otimes\calO_\infty$, then the hypothesis guarantees that the class $[\varphi]$ in $KK(C(\bbT)\otimes A,B)$ is equal to the class $[\theta\otimes\psi]$ for some unital $^*$-monomorphisms $\theta\colon C(\bbT)\to\calO_\infty$ and $\psi\colon A\to B$.
By the classification results of Kirchberg and Phillips, $\varphi$ is approximately unitarily equivalent to $\theta\otimes\psi$ and hence a path of unitaries with the desired properties
can approximately be obtained as a unitary conjugate of a path of unitaries in $1_B\otimes\calO_\infty$ connecting $1_B\otimes\theta(z)$ to $1_B\otimes1_{\calO_\infty}$.
The bound on the length of such a path follows from $\cel(\calO_\infty)=2\pi$, which is shown in \cite{Phi02}.
\end{proof}

We now come to the version of the Basic Homotopy Lemma which will be most suitable for our purposes.
In fact, what we need is an argument to shorten approximately central unitary paths.
Such a statement can be obtained from Proposition \ref{prop: KK-obstacle} as follows:
If we already assume the existence of some approximately central path (of any length) connecting the unitary to the identity, the $KK$-theoretic obstacle in \ref{prop: KK-obstacle} will automatically vanish and therefore the conclusion provides us with a new approximately central path with controlled length. 

This shortening argument is closely related to the question for the exponential length of central sequence algebras. 
We will make this precise later on in \ref{cor: central sequence algebra}.\\

Throughout the rest of this paper, we will use the following terminology:
Given a class $\calC$ of $C^*$-algebras, we say that a $C^*$-algebra $A$ can be locally approximated by elements of $\calC$ if for every finite subset $\calF$ of $A$ and every $\epsilon>0$ there exists a $C^*$-subalgebra $A_0$ of $A$ with $\dist(\calF,A_0)<\epsilon$ such that $A_0\cong C$ for some $C\in\calC$.

Note that the the class of $C^*$-algebras $A$ covered by the following theorem includes all unital AF-algebras as well as all unital Kirchberg algebras in the UCT-class.

\begin{theorem}\label{thm: shortening}
 Let $A$ be a unital $C^*$-algebra that can be locally approximated by finite direct sums of separable, unital, nuclear, simple, $KK$-semiprojective $C^*$-algebras.
 Then for any $\epsilon>0$ and every finite subset $\calF\subset A$, there are $\delta>0$ and a finite subset $\calG\subset A$ such that the following holds:
 
 Given a unital Kirchberg algebra $B$, a unital $^*$-homomorphism $\varphi\colon A\to B$ and a continuous path of unitaries $(u_t)_{t\in[0,1]}$ in $B$ with $u_0=1_B$ such that
 \[
  \left\|\left[\varphi(x),u_t\right]\right\|<\delta\quad\text{for all}\ x\in\calG\ \text{and}\ t\in[0,1],
 \]
 there exists a rectifiable continuous path of unitaries $(u'_t)$ in $B$ from $u'_0=1_B$ to $u'_1=u_1$ such that
 \[
  \left\|\left[\varphi(x),u'_t\right]\right\|<\epsilon\quad\text{for all}\ x\in\calF\ \text{and}\ t\in[0,1],
 \]
 and
 \[
  \Length\left((u'_t)_{t\in[0,1]}\right)\leq 2\pi+\epsilon.
 \]
\end{theorem}

\begin{proof}
It suffices to prove the statement for $A$ a separable, unital, nuclear, simple, $KK$-semiprojective $C^*$-algebra.
By Lemma \ref{lem: technical}, we may further assume that $u_1$ commutes with the image of $\varphi$ and hence defines a unital $^*$-homomorphism $\phi\colon C(\bbT)\otimes A\to B$ via $\phi(z\otimes 1)=u_1$ and $\phi|_{1\otimes A}=\varphi$.
Now the claim will follow from Proposition \ref{prop: KK-obstacle}, once we can show that the class $[\phi|_{SA}]$ vanishes in $KK(SA,B)$.
For this we exploit $KK$-semiprojectivity of $C(\bbT)\otimes A$, see \ref{lem: suspension}, and the existence of the homotopy $(u_t)$ as follows.

Let $\{a_n\colon n\in\bbN\}$ be a generating set for $A$ with $\lim_n \|a_n\| = 0$ and define $C(\bbT)\otimes_\delta A$ to be the universal unital $C^*$-algebra generated by a copy of $A$ and a unitary $z_\delta$ such that $\|[a_n,z_\delta]\|\leq\delta$ for all $n$.
The natural quotient maps $C(\bbT)\otimes_\delta A\to C(\bbT)\otimes_{\delta'} A$ for $\delta>\delta'$ given by $u_\delta\mapsto u_{\delta'}$ form an inductive system the limit of which we identify with $C(\bbT)\otimes A$.
By $KK$-semiprojectivity of $C(\bbT)\otimes A$, there exists $\delta_0>0$ and an element $x\in KK(C(\bbT)\otimes A,C(\bbT)\otimes_{\delta_0} A)$ with $\pi_*(x)=[\id_{C(\bbT)\otimes A}]$, where $\pi\colon C(\bbT)\otimes_{\delta_0} A\to C(\bbT)\otimes A$ denotes the canonical map.

As $\delta_0$ is independent of $\calF$ and $\epsilon$, we may assume that $\calG$ contains the finite set $\{a_n\colon \|a_n\|\geq\delta_0/2\}$ and that $\delta$ was chosen to be smaller than $\delta_0$.
Then $\|[\varphi(a_n),u_t]\|\leq\delta_0$ holds for all $n$ and $t\in[0,1]$, hence the path of unitaries $(u_t)_{t\in[0,1]}$ gives rise to a homotopy of $^*$-homomorphisms from $\phi\circ\pi$ to $(1\otimes\varphi)\circ\pi$.
This shows $[\phi]=(\phi\circ\pi)_*(x)=((1\otimes\varphi)\circ\pi)_*(x)=[1\otimes\varphi]$ in $KK(C(\bbT)\otimes A,B)$ and therefore $[\phi|_{SA}]=0$ in $KK(SA,B)$.
\end{proof}

The shortening argument of \ref{thm: shortening} can used to manipulate paths of unitaries in central sequence algebras (and algebras alike, cf.\ the proof of Lemma \ref{lem: lifting unitary paths}).
For instance, one can control the exponential length of these algebras.

\begin{corollary}\label{cor: central sequence algebra}
The central sequence algebra of any unital $KK$-semiprojective Kirchberg algebra $A$ has finite exponential length.
More precisely, we have
\[
\cel(A_\infty\cap A')=2\pi.
\]
The same statement holds for any unital Kirchberg algebra which satisfies the UCT.
\end{corollary}

\begin{proof}
Let $U_1=[(u^{(1)}_1,u^{(2)}_1,...)]$ be an element of $\calU_0(A_\infty\cap A')$ and $(U_t)_{t\in[0,1]}$ a continuous path of unitaries connecting it to $U_0=1_{A_\infty\cap A'}$.
This means we can find a sequence of continuous paths of unitaries $(u^{(n)}_t)_{t\in[0,1]}$ in $A$ from $u^{(n)}_1$ to $u^{(n)}_0=1_A$ such that $U_t=[(u^{(1)}_t,u^{(2)}_t,...)]$ for all $n$ and $t\in[0,1]$.
As the path $(U_t)$ is contained in the central sequence algebra, we have $\lim_{n\to\infty}\max_{t\in[0,1]}\|[a,u^{(n)}_t]\|=0$ for every $a\in A$. 
We may therefore apply Theorem \ref{thm: shortening} and replace the paths $(u^{(n)}_t)$ by continuous unitary paths $(\oline{u}^{(n)}_t)_{t\in[0,1]}$ from $\oline{u}^{(n)}_1=u^{(n)}_1$ to $\oline{u}^{(n)}_0=1_A$ which are still asymptotically central but further satisfy $\lim_{n\to\infty}\Length((\oline{u}^{(n)}_t)\colon t\in[0,1])\leq2\pi$.
Adding scalar-valued pieces if necessary, we may also assume $\Length((\oline{u}^{(n)}_t)\colon t\in[0,1])\geq2\pi$ for all $n$. 
After length parametrizing the paths $(\oline{u}^{(n)}_t)$, we find $\oline{U}_t=[(\oline{u}^{(1)}_{2\pi t},\oline{u}^{(2)}_{2\pi t},...)]$, $0\leq t\leq 1$, to be a now continuous path of unitaries in $A_\infty\cap A'$ from $\oline{U}_1=U_1$ to $\oline{U}_0=1_{A_\infty\cap A'}$ with $\Length((\oline{U}_t)\colon t\in[0,1])=2\pi$.

Any unital Kirchberg algebra in the UCT-class can be realized as an inductive limit over unital UCT-Kirchberg algebras with finitely generated $K$-groups, see \cite[Proposition 8.4.13]{Ror02}. As these are $KK$-semiprojective by \cite{RS87}, the claim follows from the first part of the corollary. 
\end{proof}

If the $C^*$-algebra $A$ in Proposition \ref{prop: KK-obstacle} satisfies the UCT and moreover $K_*(A)$ is finitely generated, one can use the natural isomorphism 
\[
KK(SA,B)\cong\Hom_\Lambda(\underline{K}(SA),\underline{K}(B))
\]
 from the UMCT of \cite{DL96} to give  a description of the obstruction in terms of (total) $K$-theory.
This leads to the following 'typical' version of the Basic Homotopy Lemma, which uses the language of \cite{Lin10}.
We only state it here for the sake of completeness and leave it to the reader to verify, using \cite[Proposition 2.4]{Lin10} and the above-mentioned isomorphism, that in this case the Bott obstruction defined in \cite{Lin10} coincides with the $KK$-theoretic obstacle appearing in Proposition \ref{prop: KK-obstacle}.

\begin{theorem}
 Let $A$ be a unital $C^*$-algebra that can be locally approximated by finite direct sums of separable, unital, nuclear, simple $C^*$-algebras in the UCT-class with finitely generated $K$-groups.
 Then for any $\epsilon>0$ and every finite subset $\calF\subset A$, there are $\delta>0$, a finite subset $\calG\subset A$ and a finite subset $\calP\subset\uline{K}(A)$ such that the following holds:
 
 Suppose $B$ is a unital Kirchberg algebra, $\varphi\colon A\to B$ is a unital $^*$-homomorphism and $u$ is a unitary in $B$ such that
 \[
  \|[\varphi(x),u]\|<\delta\quad\text{for all}\ x\in\calG\ \text{and}\ \Bott(\varphi,u)_{|\calP}=\{0\}.
 \]
 Then there exists a rectifiable continuous path of unitaries $(u_t)$ from $u_0=u$ to $u_1=1_B$ such that
 \[
  \|[\varphi(x),u_t]\|<\epsilon\quad\text{for all}\ x\in\calF\ \text{and}\ t\in[0,1],
 \]
 and
 \[
  \Length\left((u_t)_{t\in[0,1]}\right)\leq 2\pi+\epsilon.
 \]
\end{theorem}


\section{Lifting asymptotic unitary equivalences}\label{sec: 3}

We start this section by studying the following $C^*$-algebraic lifting problem:
Given a quotient map $B\to B/I$ and a unitary $u\in B/I$ which is approximately central (in the sense that it almost commutes with a large finite subset of $B/I$), can we find a unitary lift $\oline{u}$ of $u$ which is approximately central in $B$?

As it turns out, it is not enough to just require the unitary $u$ to be connected to the unit of $B/I$ in an approximately central way, but one needs to further put a bound on the rectifiable length of such homotopies in order to obtain a positive lifting result, see Lemma \ref{lem: lifting unitary paths}.

While homotopy arguments are known to play an important role in the study of lifting problems for $C^*$-algebras, a phenomenon best illustrated by the main result of \cite{Bla12}, the length of homotopies has not been relevant here so far.
Lemma \ref{lem: lifting unitary paths} therefore provides a new ingredient to this field, which is of particular importance for stability questions concerning commutation relations.

\begin{lemma}\label{lem: lifting unitary paths}
Let $L>0$, $n\in\bbN$ and $\epsilon>0$ be given, then there exists $\delta>0$ with the following properties:

Suppose $B$ is a unital $C^*$-algebra with ideal $I$, $\pi\colon B\to B/I$ the corresponding quotient map and $\{x_1,...,x_n\}$ a set of contractions in $B$. 
If $(u_t)_{t\in[0,1]}\subset B/I$ is a rectifiable path of unitaries of length at most $L$ with $u_0=1$ such that
 \[
  \|[\pi(x_i),u_t]\|<\delta\quad\text{for all}\ i=1,...,n\ \text{and all}\ t\in[0,1],
 \]
 then the path $(u_t)_{t\in[0,1]}$ lifts to a path of unitaries $(\oline{u}_t)_{t\in[0,1]}$ in $B$ with
 \[
  \|[x_i,\oline{u}_t]\|<\epsilon\quad\text{for all}\ i=1,...,n\ \text{and all}\ t\in[0,1].
 \]
\end{lemma}

\begin{proof}
Assume otherwise, i.e.\ assume there are $L>0$, $n\in\bbN$ and $\epsilon>0$ such that for every $m\in\bbN$ there exist a unital $C^*$-algebra $B_m$ with ideal $I_m$, a set of contractions $\{x^{(m)}_1,...,x^{(m)}_n\}\subset B_m$ and rectifiable paths $(u^{(m)}_t)_{t\in[0,1]}$ of unitaries in $B_m/I_m$ which have length bounded by $L$, commute pointwise up to $1/m$ with $\{\pi_m(x^{(m)}_1),...,\pi_m(x^{(m)}_n)\}$ and satisfy $u^{(m)}_0=1_{B_m/I_m}$, but such that there is no lift of any $(u^{(m)}_t)_{t\in[0,1]}$ to a unitary path $(\oline{u}^{(m)}_t)_{t\in[0,1]}$ in $B_m$ which satisfies $\oline{u}^{(m)}_0=1_{B_m}$ and commutes pointwise up to $\epsilon$ with $\{x^{(m)}_1,...,x^{(m)}_n\}$.

We may assume that the sets $\{x^{(m)}_1,...,x^{(m)}_n\}$ are selfadjoint for all $m$ and consider the $C^*$-algebra
\[
B:=\left\{(b_m)\colon\left[b_m,x^{(m)}_i\right]\to 0\ \text{for all}\ i=1,...,n\right\}\subset\prod_mB_m,
\]
i.e.\ the preimage of the commutant of the image of $C^*((x^{(m)}_1)_m,...,(x^{(m)}_n)_m)$ in \linebreak $\prod_mB_m/\bigoplus_m B_m$.
Writing $I=B\cap\prod_m I_m$, one uses approximate units for the ideals $I_m$ which are quasicentral for $B_m$ to see that in fact
\[
B/I=\left\{(c_m)\colon\left[c_m,\pi_m\left(x^{(m)}_i\right)\right]\to 0\ \text{for all}\ i=1,...,n\right\}\subset\prod_m(B_m/I_m).
\]
We may further assume that the paths $u^{(m)}_t$ all have length $L$ and are length parametrized.
But then the path $u_t:=(u^{(m)}_t)_m$, $t\in[0,L]$ is contained in $B/I$ and is, most importantly, continuous.
As moreover $u_0=1_{B/I}$, this path lifts to a continuous path of unitaries $\oline{u}_t$ in $B$ with $\oline{u}_0=1_B$.
Writing $\oline{u}_t=(\oline{u}^{(m)}_t)_m$, for each $m$ we find $(\oline{u}^{(m)}_t)_{t\in[0,1]}$ to be a lift of the path $(u^{(m)}_t)_{t\in[0,1]}$ satisfying $\oline{u}^{(m)}_0=1_{B_m}$.
Moreover, by definition of $B$ we have
\[
\lim_{m\to\infty}\max_{t\in[0,L]}\left\|\left[\oline{u}^{(m)}_t,x^{(m)}_i\right]\right\|=0
\]
for each $i=1,...,n$ , a contradiction.
\end{proof}

By combining Lemma \ref{lem: lifting unitary paths} with the shortening argument from Section \ref{sec: 2}, we can now lift arbitrarily long paths of unitaries while keeping track of their asymptotic behavior.
This results in a lifting procedure for asymptotic unitary equivalences of $^*$-homomorphisms as described in Theorem \ref{thm: lifting} below.

\begin{definition}\label{def: asymptotic unitary equivalence}
Let $A,B$ be $C^*$-algebras, assume that $B$ is unital.

Given two $^*$-homomorphisms $\varphi,\psi\colon A\to B$, we say that $\varphi$ and $\psi$ are asymptotically unitarily equivalent, denoted by $\varphi\approx_{\text{uh}}\psi$, if there exists a continuous path $(u_t)_{t\in[0,\infty)}$ of unitaries in $B$ such that
\[
\left\|u_t\varphi(x)u^*_t-\psi(x)\right\|\to 0
\]
as $t\to\infty$ for all $x\in A$.
If one can moreover arrange $u_0=1$, we say that $\varphi$ and $\psi$ are strongly asymptotically unitarily equivalent and write $\varphi\approx_{\text{suh}}\psi$ in this case.
\end{definition}

\begin{theorem}\label{thm: lifting}
 Let $A$ be a separable unital $C^*$-algebra that can be locally approximated by finite direct sums of separable, unital, nuclear, simple, $KK$-semiprojective $C^*$-algebras. 
 Further let $B$ be a unital Kirchberg algebra and $\varphi,\psi\colon A\to B$ two unital $^*$-homomorphisms with $\varphi\approx_{\text{suh}}\psi$. 
 Then the following holds:
 
 Suppose $D$ is a unital $C^*$-algebra and $\pi\colon D\to B$ a surjective $^*$-homomorphism.
 If $\varphi$ lifts to $D$, i.e.\ there exists a $^*$-homomorphism $\oline{\varphi}\colon A\to D$ with $\pi\circ\oline{\varphi}=\varphi$, then so does $\psi$. 
 Moreover, there exists a lift $\oline{\psi}$ of $\psi$ with $\oline{\varphi}\approx_{\text{suh}}\oline{\psi}$.
\end{theorem}

\begin{proof}
 Let $(\calF_n)$ be an increasing sequence of finite sets of contractions whose union is dense in the unit ball of $A$ and set $\epsilon_n=2^{-n}$.
 We obtain parameters $\delta_n>0$ from Lemma \ref{lem: lifting unitary paths} for the data $(L=4,|\calF_n|,\epsilon_n)$, and then find finite subsets $\calG_n$ of $A$ and parameters $\gamma_n>0$ satisfying the conclusion of Theorem \ref{thm: shortening}  for the pairs $(\calF_n,\delta_n)$.
 
 Given $\varphi$ and $\psi$ as in the statement, let $(u_t)_{t\in[0,\infty)}$ be a continuous path of unitaries in $B$ implementing the equivalence of $\varphi$ and $\psi$, i.e.\ we have $u_0=1_B$ and $\Ad(u_t)\circ\varphi$ converges pointwise to $\psi$ as $t\to\infty$.
We may reparametrize the path $(u_t)$ so that
 \[
 \left\|u_t\varphi(x)u^*_t-\psi(x)\right\|<\frac{\gamma_n}{2}\quad\text{for all}\ x\in\calG_n\ \text{and}\ t\geq n
 \]
 holds for all $n$. Since then
 \[
 \left\|\left[u_{n+t}u_n^*,(\Ad(u_n)\circ\varphi)(x)\right]\right\|=\left\|u_{n+t}\varphi(x)u^*_{n+t}-u_n\varphi(x)u^*_n\right\|<\gamma_n
 \]
 for each fixed $n$, all $x\in\calG_n$ and $t\in[0,1]$, we may apply Theorem \ref{thm: shortening} to replace $(u_t)_{t\in[n,n+1]}$ by a new, rectifiable path $(u'_t)_{t\in[n,n+1]}$ of unitaries in $B$ from $u'_n=u_n$ to $u'_{n+1}=u_{n+1}$, which commutes pointwise up to $\delta_n$ with $\calF_n$ and moreover has length bounded by 4.
But now Lemma \ref{lem: lifting unitary paths} applies iteratively, and we can lift $(u'_t)_{t\in[0,\infty)}$ to a unitary path $(v_t)_{t\in[0,\infty)}$ in $D$ satisfying $v_0=1_D$ and
 \[
 \left\|v_{n+t}\oline{\varphi}(x)v^*_{n+t}-v_n\oline{\varphi}(x)v^*_n\right\|=\left\|\left[v_{n+t}v_n^*,(\Ad(v_n)\circ\oline{\varphi})(x)\right]\right\|<\epsilon_n 
 \]
 for each fixed $n$, all $x\in\calF_n$ and $t\in[0,1]$. 
 It is clear by the choice of $\calF_n$, $\epsilon_n$ that $\Ad(v_t)\circ\oline{\varphi}$ now converges pointwise as $t\to\infty$, we denote the limit $^*$-homomorphism $A\to D$ by $\oline{\psi}$. 
 Then $\oline{\varphi}\approx_{\text{suh}}\oline{\psi}$ and 
 \[
 \pi\circ\oline{\psi}=\lim_{t\to\infty}\pi\circ(\Ad(v_t)\circ\oline{\varphi})=\lim_{t\to\infty}\Ad(u'_t)\circ\varphi=\lim_{t\to\infty}\Ad(u_t)\circ\varphi=\psi
 \]
 is verified pointwise.
\end{proof}


\section{Semiprojectivity for Kirchberg algebras}\label{sec: 4}

We now study the consequences of the lifting result from Section \ref{sec: 3} for semiprojectivity questions concerning Kirchberg algebras.
For this purpose we need to recall the notion of semiprojective $^*$-homomorphisms from \cite{Bla85}.

\begin{definition}\label{def:sp and weak sp}
A $^*$-homomorphism $\varphi\colon A\to B$ is called semiprojective if for every $C^*$-algebra $D$, every increasing chain of ideals $J_n$ in $D$ with $J_\infty=\oline{\bigcup_n J_n}$ and every $^*$-homomorphism $\sigma\colon B\to D/J_\infty$ there exists $n\in\bbN$ and a $^*$-homomorphism $\psi\colon A\to D/J_n$ making the following diagram commute:
  \[
    \xymatrix{
      && D \ar@{->>}[d] \\
      && D/J_n \ar@{->>}[d] \\
      A \ar@/^1pc/@{..>}[urr]^\psi \ar[r]^\varphi & B \ar[r]^(0.4)\sigma & D/J_\infty
    }
  \] 
We call a separable $C^*$-algebra $A$ semiprojective if the identity map $\id_A\colon A\to A$ is semiprojective.
\end{definition}

Note that by standard reduction arguments, the definition above does not change if the $^*$-homomorphism $\sigma$ is restricted to be an isomorphism.
For details and further reading on semiprojectivity we refer the reader to Loring's book \cite{Lor97}.

In this context, Theorem \ref{thm: lifting} can be regarded as a closure result for semiprojectivity of $^*$-homomorphisms between certain classes of $C^*$-algebras.
Since semiprojectivity is clearly closed under unitary conjugation, we even obtain a closure result with respect to asymptotic unitary equivalence.

\begin{lemma}\label{lem: sp closed under uh}
Let $A$ be a separable unital $C^*$-algebra that can be locally approximated by finite direct sums of separable, unital, nuclear, simple $KK$-semiprojective $C^*$-algebras and $B$ a unital Kirchberg algebra.
Then the set of unital semiprojective $^*$-homomorphisms from $A$ to $B$ is closed under asymptotic unitary equivalence.
\end{lemma}

By Kirchberg-Phillips classification (\cite{Kir}, \cite{Phi00}), two unital $ ^*$-homomorphisms between Kirchberg algebras sharing the same $KK$-class are asymptotically unitarily equivalent.
We therefore also obtain a closure result with respect to $KK$.

\begin{lemma}\label{lem: sp closed under kk}
Let $A$ and $B$ be unital Kirchberg algebras with $A$ either $KK$-semi\-projective or in the UCT-class.
Then, given two unital $^*$-homomorphisms $\varphi,\psi\colon A\to B$, the following holds:
If $\varphi$ is semiprojective and $[\varphi]=[\psi]$ in $KK(A,B)$, then $\psi$ is also semiprojective.
\end{lemma}

The above lemma brings up the question which Kirchberg algebras admit semiprojective endomorphisms that are asymptotically unitarily equivalent to the identity map.
Inspired by Neub{\"u}ser's method of showing that UCT-Kirchberg algebras with finitely generated $K$-groups are asymptotically semiprojective, \cite[Satz 6.12]{Neu00}, we show that the existence of such endomorphisms is equivalent to $KK$-semiprojectivity.

\begin{proposition}\label{prop: existence sp endo}
Let $A$ be a unital Kirchberg algebra.
Then there exists a semiprojective unital endomorphism $\alpha\colon A\to A$ with $KK(\alpha)=KK(\id_A)$ if and only if $A$ is $KK$-semiprojective.
\end{proposition}

\begin{proof}
 The 'only if' part is essentially contained in \cite[Corollary 3.13]{Dad09}.
 To verify the other implication, assume that $A$ is $KK$-semiprojective.
 By Theorem 4.3 of \cite{Bla85} there exists a shape system for $A$, i.e.\ we can write $A$ as an inductive limit $\varinjlim A_n$ with semiprojective unital connecting maps (for some unital separable $C^*$-algebras $A_n$).
 We denote by $\pi_n\colon A_n\to\varinjlim A_n$ the canonical maps, which are semiprojective as well.
 By $KK$-semiprojectivity of $A$ we can now lift $KK(\id_A)$ to an element in $KK(A,A_n)$ for some $n$.
 Using that $A_n$ must be properly infinite for $n$ sufficiently large, Kirchberg's Classification Theorem (\cite{Kir}, see \cite[Theorem 8.3.3]{Ror02}) allows us to realize this $KK$-class by a $^*$-homomorphism $\beta$ (a priori from $A$ to $A_n\otimes\bbK$, but using $KK(\pi_n\circ\beta)=KK(\id_A)$ in the uniqueness part of Kirchberg's theorem one can actually arrange $\beta$ to be a unital $^*$-homomorphism $A\to A_n$).
 Now $(\pi_n\circ\beta)\colon A\to A$ is semiprojective, as it is a composition with one factor being semiprojective, and satisfies $KK(\pi_n\circ\beta)=KK(\id_A)$ as desired.
\end{proof}

A straightforward combination of the existence result \ref{prop: existence sp endo} with the closure result \ref{lem: sp closed under uh} now yields our main theorem.

\begin{theorem}\label{thm: main}
A Kirchberg algebra is $KK$-semiprojective if and only if it is semiprojective.
\end{theorem}

\begin{proof}
Any semiprojective Kirchberg algebra is $KK$-semiprojective by \cite[Corollary 3.13]{Dad09}.

In order to verify the other implication, let $A$ be a $KK$-semiprojective Kirchberg algebra.
If $A$ is unital, Proposition \ref{prop: existence sp endo} provides us with a semiprojective unital endomorphism $\alpha$ of $A$ satisfying $KK(\alpha)=KK(\id_A)$.
Hence, by Lemma \ref{lem: sp closed under uh}, semiprojectivity passes from $\alpha$ to $\id_A$, i.e.\ $A$ is semiprojective.

If $A$ is stable, let $p$ be any nonzero projection in $A$ and consider the corner $pAp$.
Clearly, $pAp$ is also $KK$-semiprojective and hence semiprojective by the first part.
Using \cite[Theorem 4.1]{Bla04}, it follows that $A\cong pAp\otimes\bbK$ is semiprojective as well.
\end{proof}

Since a Kirchberg algebra satisfying the UCT is $KK$-semiprojective if and only if it has finitely generated $K$-groups (\cite[Proposition 3.14]{Dad09}, cf.\ also \cite[Corollary 2.11]{Bla04}), Theorem \ref{thm: main} gives a positive solution to \hyperref[conj:Blackadar]{Blackadar's Conjecture}.

\begin{corollary}\label{cor:Blackadar conjecture}
 A UCT-Kirchberg algebra is semiprojective if and only if it has finitely generated $K$-theory.
\end{corollary}

One may ask in general, i.e.\ outside the finitely generated case, which unital $^*$-homomorphisms between UCT-Kirchberg algebras $A$ and $B$ are semiprojective.
By Lemma \ref{lem: sp closed under kk}, we know that this only depends on the $KK$-class of the homomorphism.
In some cases this is still determined by $K$-theory, as in the following corollary which can be proved along the lines of $\ref{prop: existence sp endo}$ and $\ref{thm: main}$.

\begin{corollary}\label{cor: fg image}
Let $A$ and $B$ be unital UCT-Kirchberg algebras.
If $K_*(A)$ is free or $K_*(B)$ is divisible, then a unital $^*$-homomorphism $\varphi\colon A\to B$ is semiprojective if and only if its image on $K$-theory is finitely generated.
\end{corollary}

However, having finitely generated image on $K$-theory is in general not enough to ensure semiprojectivity of a $^*$-homomorphism between UCT-Kirchberg algebras.
The reason being that the class of any semiprojective $^*$-homomorphism $\varphi\colon A\to\varinjlim B_n$ must lift to an element in $\varinjlim KK(A,B_n)$ by \cite[Theorem 3.1]{Bla85}.
But as the functor $\Ext^1_\bbZ(K_*(A),\cdot)$ is not continuous in general, one can even exhibit examples of Kirchberg algebras $A,B_n$ and $^*$-homomorphisms $\varphi\colon A\to\varinjlim B_n$ with vanishing $K_*(\varphi)$ for which the class $[\varphi]$ does not lift as described above.


\end{document}